\documentclass[table]{tMAM2e}

\usepackage{subfigure}
\usepackage{dsfont}
\usepackage{stmaryrd}

\usepackage{tikz}
\usetikzlibrary{calc}
\usepackage{calc}
\usepackage{algorithmic}
\usepackage{mathtools}

\usepackage{color}

\definecolor{lightgray}{gray}{0.9}
 	
\usepackage{hyperref} 
\usepackage{hypcap} 

\usepackage{caption}

\theoremstyle{plain}
\newtheorem{theorem}{Theorem}[section]
\newtheorem{corollary}[theorem]{Corollary}
\newtheorem{definition}[theorem]{Definition}
\newtheorem{lemma}[theorem]{Lemma}
\newtheorem{proposition}[theorem]{Proposition}

\theoremstyle{remark}
\newtheorem{remark}[theorem]{Remark}
\newtheorem{example}[theorem]{Example}

\begin{document}

\title{{\itshape Journal of Mathematics and Music} From covering to tiling modulus $p$\break (Modulus $p$ Vuza canons: generalities and resolution of the case $\left\lbrace 0,1,2^k\right\rbrace $ with $p =2$.)}

\author{Helianthe Caure$^{\rm a}$ $^{\ast}$\thanks{$^\ast$Corresponding author. Email: Helianthe.Caure@ircam.fr}
\\\vspace{6pt}  $^{a}${Institut de Recherche et Coordination Acoustique/Musique};
\\\received{v1.0 released December 2013} }

\maketitle

\begin{abstract}

Non-periodic tilings of $\mathbb{Z}_N$ are a difficult to obtain key to a wide range of complex mathematical issues. In a musician's eyes, they are called rhythmic Vuza canons. After a short summary of the main properties of rhythmic tiling canons and after stressing the importance of Vuza canons, this article presents a new approach: modulo $p$ Vuza canons. Working modulo $p$ (especially with $p=2$) allows Vuza canons to be computed very quickly. But going back to traditional Vuza canons requires understanding of when and how the tiling is a covering. The article's main result is a complete case construction of a modulo $2$ tiling. This first solution of a modulo $2$ tiling is proved in a constructive way, which, it is hoped, is expandable to every modulo $p$ tiling. Following this path may lead to a solution to the problem of building modulo $p$ tilings and thus a solution to the problem of easily finding traditional Vuza canons.


\end{abstract}

\section{Introduction}

\subsection{Context}

The science of rhythmic tiling canons came into being thanks to both Dan Tudor Vuza \citeyearpar{vuza1991supplementary}, who approached the subject from a musical point of view, and Gy{\"o}rgy Haj{\'o}s \citeyearpar{hajos1950factorisation}, who previously developed a mathematical theory -the notion of finite abelian group factorization (while he was solving the Minkowski conjecture)- that was later linked to Vuza's work by Moreno Andreatta \citeyearpar{andreatta1996gruppi}.

Musically, a rhythmic tiling canon is a canon where only the onsets matter. A musical note is usually described by its pitch, onset, offset, velocity...  The canon is said rhythmic because most of the note classical attributes are not relevant: it is a canon where the main voice is a pattern of onsets. The tiling property of the canon comes from the constraint that the pattern of onsets and its translates -the other voices of the canon- must tile the time axis: one must hear at every beat one and only one onset.

Thanks to Vuza's modelisation of rhythm \citep{vuza1985rythme}, such a canon is equivalent to a tiling of the discrete line $\mathbb{Z}$ by a finite integers set and its translates. This, as we will explain later, can be represented by a direct sum of a two integers sets being equal to $\mathbb{Z}$.

Lagarias and Wang have proved in a more general case \citep[theorem 5]{lagarias1996tiling} that a tiling of the line by translations is periodic. This result was already known by Haj{\'o}s and DeBruijn \citep{hajos1950factorisation}, although they never proved it in details. Hence, for modelling rhythmic tiling canons, we will not focus on a tiling of $\mathbb{Z}$ by a finite set of integers, but on a tiling of $\mathbb{Z}_N$.

For a more complete history of rhythmic canons and recent research on the topic, the reader is invited to refer to the special issues on the topic  which appeared in both Journal of Mathematics $\&$ Music \citeyearpar{JMMtiling} and Perspectives of New Music \citeyearpar{PNMtiling}.

\subsection{Model}

\begin{definition}
We will denote equivalently a rhythmic pattern by
\begin{enumerate}
\item a finite set $A$ of non negative integers containing 0
\item a polynomial $A(X) \in \left\lbrace 0,1\right\rbrace \left[ X\right] $
\item an infinite word with finitely many non-zero letters or a finite word $A_{0-1}$ over the alphabet $\left\lbrace 0,1\right\rbrace $.
\end{enumerate}

By denoting for any word $u$, the subword of length $k$ starting at the $n$-th letter  $u\left[ n, \ldots, n+k-1\right] $  $\forall k\in \mathbb{N}¨^*$, the equivalent notations are linked by

 $A(X) = \sum_{a\in A} X^a$ and $\forall n \in \mathbb{N} \   A_{0-1}\left[ n\right] = \mathds{1}_A(n)$, where $\mathds{1}_A$ is the characteristic function of $A$ and its length verifies $\max A+1\leq \vert A_{0-1} \vert \leq \infty $.

Elements of a rhythmic pattern are called onsets.
\end{definition}

\begin{definition}
Let $\divideontimes$ be a standard set-theoretic operation (like the sum or cartesian product)  between sets of integers, and $N\in \mathbb{N}$. We note $\divideontimes_N$ the operation defined by $A\divideontimes_N B = \left\lbrace a \divideontimes b \text{ mod } N, a\in A, b\in B\right\rbrace $.

\end{definition}

\begin{definition}\label{deftil}
Let $N \in \mathbb{N}^*$, the rhythmic pattern $A$ tiles $\mathbb{Z}_N$ if there exists another rhythmic pattern $B$ such that one of the following equivalent conditions is satisfied:
\begin{enumerate}
\item $A \oplus_N B = \mathbb{Z}_N $ where $\oplus$ is the direct sum (or injective sum $\oplus : A\times B \rightarrow \mathbb{Z}$)
\item $A(X) \cdot B(X) = 1 + X + \ldots + X^{N-1} \text{ mod } (X^N -1)$
\item $\mathds{1}_A \star \mathds{1}_B = \mathds{1}_{\mathbb{Z_N}}$ where $\star$ is the convolution product.
\end{enumerate}

We say that $(A,B)$ is a rhythmic tiling canon (RTC) of $\mathbb{Z}_N$, or that the pattern $A$ tiles $\mathbb{Z}_N$ with the (pattern of) entries $B$.
\end{definition}

\begin{remark}
It is important to understand that a rhythmic pattern $A$ that tiles $\mathbb{Z}_N$  doesn't have to be included in $\left\lbrace 0,\ldots, N-1\right\rbrace $. Indeed, it is possible that there exists $a \in A$ such that $N <a$, because only the elements of the sum $A+B$ are projected in $\mathbb{Z}_N$. The reader can even prove that from a RTC $(A,B)$ of $\mathbb{Z}_N$, one can built $A'$ verifying this latter condition such that $(A',B)$ is a RTC of $\mathbb{Z}_N$.
\end{remark}

\begin{example}
A rhythmic pattern can be written in the following ways
\begin{enumerate}
\item $A = \left\lbrace 0,1,4,5\right\rbrace $
\item $A(X) = 1+X+X^4+X^5$
\item $A_{0-1} = 110011$ or $A_{0-1} = 1100110000$ or $A_{0-1} = 1100110\ldots$
\end{enumerate}

It tiles $\mathbb{Z}_{16}$ with $B = \left\lbrace 0,2,8,10\right\rbrace $.

\end{example}

\begin{remark}
From now on, instead of writing $n$ times a letter $a$, we will denote $\overline{a}^n = aa\cdots a$.
\end{remark}

An easy way to represent a RTC $(A,B)$, which helps ``seeing" the rhythm, is on a grid with its x-axis corresponding to the time, and its y-axis the entrance of new voices from $B$ playing the pattern $A$. An onset of a note $a+b \in A \oplus_N B$ of the RTC is depicted by a black square in position $(a+b \text{ mod } N, i)$, with $B\left[ i\right] =b$.

\begin{example}\label{rtc1}
One can represent the RTC $(\left\lbrace 0,1,4,5\right\rbrace,  \left\lbrace 0,2,8,10\right\rbrace )$ of $\mathbb{Z}_{16}$ in the following way:
\begin{center}

 \begin{tikzpicture}[scale = 0.5]
\draw (0,0) grid[xstep=1,ystep=1] (16,4);

 \fill (0,0) rectangle (2,1);
  \fill (2,1) rectangle (4,2);
  \fill (4,0) rectangle (6,1);
  \fill (6,1) rectangle (8,2);
  
  \fill (8,2) rectangle (10,3);
  \fill (10,3) rectangle (12,4);
  \fill (12,2) rectangle (14,3);
  \fill (14,3) rectangle (16,4);

   \node (N) at (0,0) {};
   \draw (N.south west) node[below]{$\left( 0,0\right) $};

\end{tikzpicture}
\end{center}
\end{example}

\begin{definition}
A RTC $(A,B)$ is \textit{compact} if $A \oplus_N B = A\oplus B = \llbracket 0,N-1 \rrbracket$ exactly, without need of projection in $\mathbb{Z}_N$.
\end{definition}

\begin{example}\label{rtc2}
$(\left\lbrace 0,1,4,5\right\rbrace,  \left\lbrace 0,2\right\rbrace )$ is compact, whereas $(\left\{0,1,3,6\right\} , \left\{0,4\right\} )$ is not.
\end{example}

\begin{definition}
Let $k \in \mathbb{N}^*$, and $A$ be a rhythmic pattern which tiles $\mathbb{Z}_N$. We note $\overline{A}^k = A \oplus \left\lbrace 0, N, 2N, \ldots, (k-1)N \right\rbrace $ its $k$-concatenation.
\end{definition}

As shown by \cite{amiot2005rhythmic} some transformations of RTC give other RTC, see:
\begin{proposition}[Duality] 

$(A,B)$ is a RTC of $\mathbb{Z}_N$ iff $(B,A)$ is a RTC of $\mathbb{Z}_N$

\end{proposition}

\begin{proposition}[Concatenation] 

Let $k \in \mathbb{N}^*$, $(A,B)$ is a RTC of $\mathbb{Z}_N$ iff $(\overline{A}^k,B)$ is a RTC of $\mathbb{Z}_{kN}$.

\end{proposition}

\begin{example}
The dual of the canon depicted on example \ref{rtc1} is the 2-concatenation of the dual of the first canon of example \ref{rtc2}.
\end{example}

\subsection{Vuza Canons}

\begin{definition}
A RTC $(A,B)$ of $\mathbb{Z}_N$ is periodic if there exists $0<k<N$ such that $A+_N \left\lbrace k\right\rbrace  = A$ or $B+_N \left\lbrace k\right\rbrace  = B$.

If a canon is not periodic, it is called a Vuza canon (VC).
\end{definition}

\begin{theorem}[\cite{amiot2005rhythmic}]\label{vuzaconcat}
Every RTC can be deduced by concatenation and duality transformations from VC and the trivial canon $\left( \left\lbrace 0\right\rbrace , \left\lbrace 0\right\rbrace\right)  $.

\end{theorem}

\begin{example}
$\left\lbrace 0,1,4,5\right\rbrace \oplus \left\lbrace 0,2\right\rbrace = \mathbb{Z}_8 $ is concatenated from $\left\lbrace 0,1\right\rbrace \oplus \left\lbrace 0,2\right\rbrace = \mathbb{Z}_4 $, whose dual is concatenated from the dual of $\left\lbrace 0,1\right\rbrace \oplus \left\lbrace 0\right\rbrace = \mathbb{Z}_2 $, concatenated from the trivial canon.

 \begin{tikzpicture}[scale = 0.5]
\draw (0,0) grid[xstep=1,ystep=1] (8,2);
 \fill (0,0) rectangle (2,1);
  \fill (2,1) rectangle (4,2);
  \fill (4,0) rectangle (6,1);
  \fill (6,1) rectangle (8,2);
  \draw[color = red, thick] (4,0) -- (4,2);
  
  \end{tikzpicture} concatenated from  \begin{tikzpicture}[scale = 0.5]
\draw (0,0) grid[xstep=1,ystep=1] (4,2);
 \fill (0,0) rectangle (2,1);
  \fill (2,1) rectangle (4,2);
  \end{tikzpicture} dual of   \begin{tikzpicture}[scale = 0.5]
\draw (0,0) grid[xstep=1,ystep=1] (4,2);
 \fill (0,0) rectangle (1,1);
  \fill (2,0) rectangle (3,1);
  \fill (1,1) rectangle (2,2);
  \fill (3,1) rectangle (4,2);
    \draw[color = red, thick] (2,0) -- (2,2);
  
  \end{tikzpicture} 
  
  concatenated from \begin{tikzpicture}[scale = 0.5]
\draw (0,0) grid[xstep=1,ystep=1] (2,2);
 \fill (0,0) rectangle (1,1);
  \fill (1,1) rectangle (2,2);
      \draw[color = red, thick] (1,0) -- (1,2);
  \end{tikzpicture} concatenated from the trivial canon: \begin{tikzpicture}[scale = 0.5]
 \fill (0,0) rectangle (1,1);

  \end{tikzpicture} 
\end{example}

\begin{remark}

It means that VC are the basis of RTC under concatenation operation. Thus, if one manages to obtain every VC, one has a constructive way to get every RTC. Note that in practice, it is difficult to built RTC from VC, even if we know that they exist and the size of their tile:

\end{remark}

\begin{theorem}

There exist VC of $\mathbb{Z}_N$ for, and only for $N$ not of the form:
$$N = p^{\alpha}, N = p^{\alpha} q, N = p^2q^2, N = pqr,  N = p^2qr, N = pqrs$$ with $p, q, r, s$ distinct primes.

\end{theorem}

\begin{proof}
This full result has been shown step by step by \cite{hajos1950factorisation}, \cite{redei1950beitrag}, \cite{de1953factorization}, \cite{sands1957factorisation} and later by \cite{vuza1985rythme} independently.
\end{proof}

Those non-periodic canons are also the key to understanding whether a given rhythmic pattern tiles or not:

The tiling condition $$A(X) \cdot B(X) = 1 + X + \ldots + X^{N-1} = \dfrac{X^N-1}{X-1}\text{ mod } (X^N -1)$$

gives the idea of splitting the $N$-th roots of unity between the factors $A$ and $B$, hence, looking for their cyclotomic factors.

\begin{definition}
Let $A$ be a rhythmic pattern, we note $$R_A = \left\lbrace d\in \mathbb{N}^*, \   \text{the $d$-th cyclotomic polynomial divides } A(X) \right\rbrace,$$ $$S_A = \left\lbrace p^a \in R_A,  p \text{ prime}, a \in \mathbb{N}^* \right\rbrace .$$
\end{definition}

\begin{definition}
Let $A$ be a rhythmic pattern, the following three conditions are called Coven-Meyerowitz conditions:

$(T_0): A$ tiles

$(T_1): A(1) = \prod _{p^\alpha \in S_A} p$

$(T_2): \text{If } p_1^\alpha, p_2^\beta, \ldots , p_r^\gamma \in S_A \text{ then } p_1^\alpha \cdot p_2^\beta \cdot \ldots  \cdot p_r^\gamma \in R_A $ with $p_i$ distinct primes.
\end{definition}

They were introduced by \cite{coven1999tiling} in the currently best attempt to get a necessary and sufficient condition to know if whether a rhythmic pattern can tile or not.

\begin{theorem}[\cite{coven1999tiling}] \label{CM}
$\ $
\begin{enumerate}
\item $(T_0) \Rightarrow (T_1)$
\item $(T_1) \wedge (T_2) \Rightarrow (T_0)$
\item If $\sharp A$ has at most two prime factors, then $(T_0) \Rightarrow(T_1) \wedge (T_2) $
\end{enumerate}

\end{theorem}

\begin{proposition}[\cite{amiot2005rhythmic}]
Condition $(T_2)$ is closed under duality and concatenation transformations.

\end{proposition}

\begin{corollary}\label{t2vc}
We have $(T_0) \Rightarrow (T_2)$ for every RTC iff $(T_0) \Rightarrow (T_2)$ for every VC.
\end{corollary}

\begin{remark}
This corollary emphasises the importance of VC. Those non-periodic canons are nonetheless the construction basis of any RTC, but they are also the bottom line to understand if a given rhythmic canon would tile.

 Though, enumeration of VC is still out of reach of mathematicians: the best we currently have is an exploration algorithm to obtain slowly all the VC with a given period. It was improved recently by \cite{KolMatoJMM}, but it is still exponential.

The next section 	endeavours to look at VC in a more complex space, which makes them easier to obtain, with a still distant hope of going back to VC of $\mathbb{Z}_N$ in a non-exponential way.
\end{remark}

\section{Modulus $p$ tiling}

Many RTC properties are formulated in polynomial notation (like \ref{CM}), and the ring $\mathbb{Z}\left[ X\right] /(X^N -1)$ seems fitted for their study. However, this ring is not factorial, hence, one cannot use the decomposition in irreducible elements of the complete polynomial $1+X+ \ldots + X^{N-1}$ to obtain all the factors in $\left\lbrace 0,1\right\rbrace \left[ X\right] $.

The polynomial notation of the tiling condition (2nd point or \ref{deftil}) gives the idea, introduced in \citep{amiot2004rhythmic} to work in the set $\left\lbrace 0,1\right\rbrace \left[ X\right] $ of polynomials with coefficients 0 or 1, but it is not a ring. Thus, we try to work in the ring $\mathbb{F}_2\left[ X\right] $ so that we have 

$$\begin{array}{ccc}
\mathbb{F}_2 \left[ X\right] \times \mathbb{F}_2 \left[ X\right] & \longrightarrow & \mathbb{F}_2 \left[ X\right]\\
(A,B)&\longmapsto & A \cdot B
\end{array}$$

and the tiling condition becomes the divisibility of $1+X+\ldots +X^{N-1}$ by  $A(X) \text{  mod  } (X^N -1,2)$, i.e. the classic divisibility condition in $\mathbb{Z}\left[ X\right] /(X^N -1, 2) = \mathbb{F}_2\left[ X\right] /(X^N -1)$, where $(X^N -1, 2)$ denotes the ideal generated by the set $\left\lbrace 2,X^N-1\right\rbrace \subset \mathbb{Z}\left[ X\right] $.

We then extend the idea to the ring $\mathbb{F}_p\left[ X\right]$, $p$ prime, and we obtain the notion of tiling modulo $p$.

\begin{definition}
$(A,B)$ is a RTC of $\mathbb{Z}_N$ modulo $p$ if
$$A(X) \cdot B(X) = 1+ X + \cdots  +X^{N-1} \text{  mod  } (X^N -1,p).$$

\end{definition}

Modulo $p$ tiling  is relevant from a mathematical point of view, but also musically, because it enriches RTC with harmony, allowing notes superposition. From now, we will denote RT$_p$C for a modulo $p$ RTC.

\begin{remark}
The equivalent notations with sets from RTC are adaptable with multisets to RT$_p$C:

Let $A$ and $B$ two rhythmic patterns, $(A,B)$ is a RT$p$C of $\mathbb{Z}_N$ if the multiset $C~=~A~+_N~B~=~\left\lbrace a+b \text{ mod } N, a \in A, b \in B \right\rbrace $ verifies $\mathds{1}_C(n) \equiv 1 \text{ mod } p$ if $  n \in \left\lbrace 0,\ldots , N-1\right\rbrace $, and $\mathds{1}_C(n) = 0$ elsewhere, with $\mathds{1}_C$ being the multiset characteristic function.

\end{remark}

\begin{remark}
When we consider a finite multiset $C$, we can also adapt the notation $C_{0-1}$ to RT$_p$C:

Let $C$ be a finite multiset of non-negative integers, we note $C_{0-p}$ the infinite word with finitely many non-zero letters or the finite word, over the alphabet $\left\lbrace 0, \ldots ,p-1\right\rbrace $, verifying  $\forall n \in \mathbb{N}, \ C_{0-p}\left[ n\right] = \mathds{1}_C(n) \text{ mod } p$ and $\max C+1\leq \vert C_{0-p} \vert \leq \infty $.

To avoid the confusion between the classical notation $C_{0-1}$ and $C_{0-2}$ for a RT$_2$C, we do not use the notation $C_{0-(p-1)}$ for this multiset, even if it seems more fitting.
\end{remark}

\begin{example}
$(\left\lbrace 0,1,3\right\rbrace ,\left\lbrace 0,2,3\right\rbrace )$ is a RT$_2$C of $\mathbb{Z}_7$:

\begin{enumerate}
\item $A+_7 B = \left\lbrace 0,1,2,3,3,3,4,5,6\right\rbrace $
\item $A(X)\cdot B(X) = 1+X+\ldots+X^6 \text{ mod } (X^7-1, 2)$
 \item $(A+B)_{0-2} = 1111111 = 111111100000\ldots$
\end{enumerate}

\end{example}

\begin{remark}
For any rhythmic pattern $A$, we have that $A_{0-1} = A_{0-p}$.

\end{remark}

\begin{definition}
Let $A$ and $B$ two rhythmic patterns that could tile or not. Since they are non empty (both contains $0$), the word $(A + B)_{0-p} $ starts with a certain number of $1$s: the beginning of the tiling. We note $m$ the first index where $(A, B)$ does not tile, i.e. the smallest index $m$ such that  $(A + B)_{0-p} \left[ m\right] \neq 1 \text{ mod } p$.

Since both rhythmic patterns are finite, by definition, the word $(A + B)_{0-p} $ is finite with (a finite or infinite number of) $0$s at the end. We note $M$ the last index such that the last onset of $(A + B)$ is at time $M$, i.e. the largest index $M$  such that $(A + B)_{0-p} \neq 0$.

\end{definition}

\begin{definition}
With previous notations, a modulo $p$ under-cover $UC_p$ of $(A,B)$ is a finite subsequence of $(A + B)_{0-p} $ of length at least $M-m+1$,  such that 
$$UC_p\left[ n\right] =(A + B)_{0-p} [n+m]$$ or equivalently this word starting with a number $u$ of $1$s, $0 \leq u < m$, or ending with a finite number of $0$s. We say that this under-cover starts at the index $ i = m-u$ and when relevant, we note $UC_p (i)$.

\end{definition}

\begin{example}
Let $A = \left\lbrace 0,1,3\right\rbrace $ and $B = \left\lbrace 0,1,4\right\rbrace $, then $A + B = \left\lbrace 0,1,1,2,3,4,4,5,7\right\rbrace $, and a modulo 3 under-cover starting at index 1 is $UC_3(1) = 2112101$ or $UC_3 = 2112101000$. 

\begin{center}
\includegraphics[scale=1]{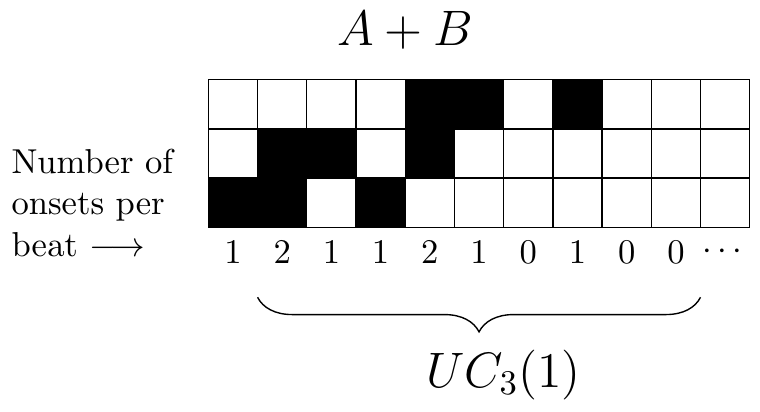}
\end{center}
\end{example}

\begin{remark}
An under-cover of $(A,B)$ is uniquely defined by its length and starting index.

Its length has to be at least $M-m+1$ so that the under-cover provides all the information needed to fill it (in the sense of definition \ref{fill}) into an actual tiling.
\end{remark}

\begin{remark}
If $(A,B)$ is a RT$_p$C, then any under-cover will be of the form $\overline{1}^k \overline{0}^l$, with $k,l \in \mathbb{N}$, namely the empty word if $k=l=0$.
\end{remark}

\begin{definition}\label{fill}

Let $A$ and $B$ two rhythmic patterns, $UC_p$ the under-cover of $(A,B)$ starting at index $i$ and of length $n = \vert UC_p \vert$.

We say that we fill this under-cover with $B' \subset \mathbb{N}$ finite if $$\left(A+\left(B \bigcup \left( B' + \left\lbrace i\right\rbrace \right) \right)\right)_{0-p}\left[ i, \ldots, i +n-1\right]  = \overline{1 }^n.$$

\begin{remark}
Note that $B'$ does not depend on $B$ in the sense that if $(A,B_1)$ and $(A,B_2)$ give the same under-cover:  $UC_p(i)_1 = UC_p(i)_2$, we will need the same $B'$ to fill it. But the minimal size of an under-cover depends on both $A$ and $B$.

\end{remark}

\end{definition}

\begin{example}\label{exp}
Let $A = \left\lbrace 0,1,4\right\rbrace $ and $B =\left\lbrace  0,2,5\right\rbrace $. We have $A+B = \left\lbrace0,1,2,3,4,5,6,6,9 \right\rbrace $, hence $(A + B)_{0-2}  = 1111110001 \left( =1111112001\right) $ and an under-cover starting at index 6 is $UC_2(6) = 0001$. We can fill this under-cover with $B' = \left\lbrace 0,2,3\right\rbrace $.

Indeed, we have $ A + \left(B \cup\left(  B' + \left\lbrace 6\right\rbrace \right) \right) = \left\lbrace0,1,2,3,4,5,6,6,6,7,8,9,9,9,10,10,12,13 \right\rbrace $, thus $ (A + (B \cup (B' + \left\lbrace 6\right\rbrace)))_{0-2} \left[ 6,9\right] =1111$.

Note that if we had chosen a longer under-cover, see $\widetilde{UC_2}(6) = 000100$, it would have been filled with $\widetilde{B}' = \left\lbrace 0,2,3,4\right\rbrace $ and not $B'$, because  $ (A + (B \cup( \widetilde{B}' + \left\lbrace 6\right\rbrace)))_{0-2} \left[ 6,11\right] =111111$ and $ (A + (B \cup (B' + \left\lbrace 6\right\rbrace)))_{0-2} \left[ 6,11\right] =111100\neq \overline{1}^6$.
\end{example}

\begin{center}
\includegraphics[scale=0.8]{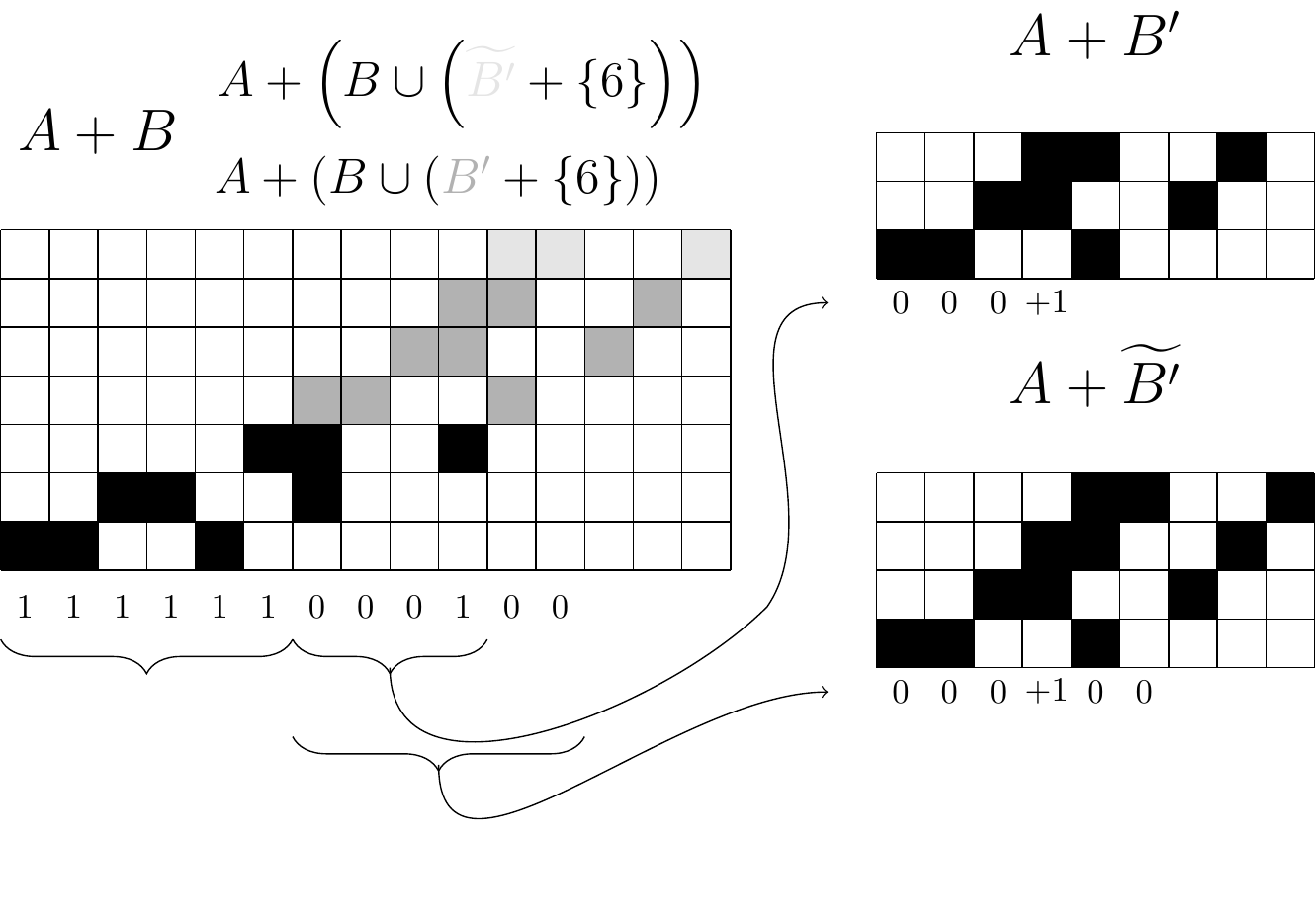}
\end{center}

\begin{remark}\label{covernul}
When one wants to tile with a pattern $A$, one can imagine filling the null under-cover $\overline{0}^k$ for some $k\in\mathbb{N}$, starting at index $0$.

This is the meaning of an under-cover: it can be seen as a partial covering that one needs to fill to obtain a tiling. This is how the greedy algorithm presented below works. 
\end{remark}

\begin{definition}\label{follow}
If one fills an under-cover $UC_p(i)$ of $(A,B)$ with $B'$, then we call ``following under-cover" an under-cover of $(A,B\cup( B' + \left\lbrace i\right\rbrace ))$, starting at index $i + \vert UC_p(i) \vert $. It will be denoted $UC_p'$.
\end{definition}

\begin{example}
In the previous example \ref{exp}, a following under-cover of $UC_2(6)$ is $UC'_2(10) = 0011$.
\end{example}

\begin{definition}
$(A,B)$ is a compact RT$_p$C of $\mathbb{Z}_N$ if $$A(X)\cdot B(X) = 1+X+\ldots + X^{N-1} \text{ mod } p.$$
\end{definition}

\begin{theorem} \citep{warusfel1971structures}
Let $p$ prime and $P \in \mathbb{F}_p\left[ X\right] $ such that $P(0) \neq 0$, then there exists $N \in \mathbb{N}$ such that $P$ divides $X^N -1$.
\end{theorem}

\begin{remark}\label{compactmod2}
As first noticed by \cite{amiot2005rhythmic}, this theorem entails that any rhythmic pattern tiles modulo $p$:

Let $A(X) \in \mathbb{Z}\left[ X\right] $ be a rhythmic pattern and a fixed $p$ prime. Considering the polynomial $P(X) = A(X)(X-1) \in \mathbb{F}_p\left[ X\right] $, and the quotient polynomial $Q \in \mathbb{F}_p\left[ X\right]$ given by the theorem such that $P(X) \cdot Q(X) = (X^N - 1) \text{ mod } p$ we can obtain the tiling condition $P(X) \cdot \widetilde{Q}(X) = (X^N - 1) \text{ mod }(X^N -1, p)$ with $\widetilde{Q}$ a polynomial representing a rhythmic pattern. Indeed, by iterating the transformation  $$\alpha X^k = (\alpha -1)X^k + X^{k+N} \text{ mod } (X^N -1)$$ over the polynomial $Q$, it can be associated to a rhythmic pattern $\widetilde{Q} \in \mathbb{Z}\left[ X\right] $, i.e. a polynomial of $\left\lbrace 0,1\right\rbrace \left[ X\right] $.

Note that in $\mathbb{F}_2\left[ X\right] $, every polynomial has its coefficients in $\left\lbrace 0,1\right\rbrace $ which means that this transformation is unnecessary, hence every rhythmic pattern gives a compact RT$_2$C. Indeed, it is clear that the non-compacity of a RTC (in $\mathbb{Z}$ or $\mathbb{F}_p$) comes from the projection modulo $(X^N -1)$; i.e. here from the transformation we apply on $Q$.
\end{remark}

\begin{example}

In $\mathbb{F}_3\left[ X\right] $, let $A(X) = 1+X+ X^3$ be a rhythmic pattern. It verifies

\begin{eqnarray*}
X^{24} -1 &=&A(X)(X-1) \cdot (X^{20} + X^{19} + 2X^{17} + 2X^{15} + 2X^{14} + 2X^{13}  \\
&& + 2X^{10} + X^9 + 2X^8 +X^7 + X^6 + X^5 + 2X^4 + 2X^3 + X^2 + 1) 
\end{eqnarray*}
 
The quotient polynomial $Q(X) = X^{20} + X^{19} + 2X^{17} + 2X^{15} + 2X^{14} + 2X^{13} + 2X^{10} + X^9 + 2X^8 +X^7 + X^6 + X^5 + 2X^4 + 2X^3 + X^2 + 1$ does not have its coefficients in $\left\lbrace 0,1\right\rbrace $, thus is not a rhythmic pattern. By applying the transformation $2X^k = X^k + X^{k+24}$ to its terms $(2*X^{17}, 2*X^{15}, 2*X^{14}, 2*X^{13}, 2*X^{10} , 2*X^8, 2*X^4 , 2*X^3)$, we obtain the polynomial  
 
\begin{eqnarray*}
 \widetilde{Q}(X) &=& X^{20} + X^{19} + X^{41} + X^{17} + X^{39} + X^{15} + X^{38} + X^{14} + X^{37}  \\
&& + X^{13} + X^{34} + X^{10}+ X^9 + X^{32} + X^8 +X^7 + X^6 + X^5 + X^{28}  \\
&& + X^4+ X^{27} + X^3 + X^2 + 1
\end{eqnarray*}

  which is a rhythmic pattern verifying
   $$A(X)\cdot \widetilde{Q}(X) = 1+X+\ldots + X^{23} \text{ mod } (X^{24} -1,3).$$

\end{example}

This theorem has inspired a modulo $p$ tiling algorithm in \citep{amiotPNM49}. We present here an algorithm devised from the previous one, but improved inasmuch as it is greedy and optimal for any rhythmic pattern which gives a compact RT$_p$C, a fortiori for any rhythmic pattern with $p=2$ thanks to remark \ref{compactmod2}.

\begin{algorithmic}[1]
\REQUIRE $A \subset \mathbb{N}$ finite and $0 \in A$
\STATE $B = \left\lbrace 0\right\rbrace $, $N = \max A$
\WHILE{$(A+B)_{0-p} \neq \overline{1}^n$ for some $n$ }
\STATE $i \leftarrow $ first index such that $(A+B)_{0-p} \neq 1 \text{ mod } p$
\IF{$ i \neq (\max B + \max A +1)$} 
\STATE $B := B \cup \left\lbrace i\right\rbrace$
\STATE $N := i+ \max A$
\ELSE 
\STATE break while
 \ENDIF
 \ENDWHILE
 \RETURN $(B,N)$
\end{algorithmic}

\begin{example}\label{A014}
This algorithm is really easy to understand from a graphical point of view. Indeed, if one considers the representation in a grid, it means adding the rhythmic pattern $A$ at the earliest time when there is not $1$ mod $p$ onset. It is also a good way to understand the notions of definitions \ref{fill} and \ref{follow}.

See the algorithm applied to $A = \left\lbrace 0,1,4\right\rbrace $. For better understanding, at every step depicted of the algorithm, we will show the value of $(A+B)_{0-2}$ and of the smaller uner-cover $UC_2(i)$ the algorithm is trying to fill.

Step 1: 

\begin{tikzpicture}[scale = 0.5]
\draw (0,0) grid[xstep=1,ystep=1] (15,7);
\fill (0,0) rectangle (2,1);
\fill (4,0) rectangle (5,1);

   \node [anchor=west] at (17,4) {\large $(A+B)_{0-2} = 11001000\cdots$};
      \node[anchor=west]  at (17,3) {\large $UC_2(2) = 001$};
\end{tikzpicture}

Loop while with $i = 2$: 

\begin{tikzpicture}[scale = 0.5]
\draw (0,0) grid[xstep=1,ystep=1] (15,7);
\fill (0,0) rectangle (2,1);
\fill (4,0) rectangle (5,1);
\fill (2,1) rectangle (4,2);
\fill (6,1) rectangle (7,2);

   \node [anchor=west] at (17,4) {\large $(A+B)_{0-2} = 1111101$};
      \node[anchor=west]  at (17,3) {\large $UC_2(5) = 01$};
\end{tikzpicture}

With $i=5$: 

\begin{tikzpicture}[scale = 0.5]
\draw (0,0) grid[xstep=1,ystep=1] (15,7);
\fill (0,0) rectangle (2,1);
\fill (4,0) rectangle (5,1);
\fill (2,1) rectangle (4,2);
\fill (6,1) rectangle (7,2);
\fill (5,2) rectangle (7,3);
\fill (9,2) rectangle (10,3);

   \node [anchor=west] at (17,4) {\large $(A+B)_{0-2} = 1111110001$};
      \node[anchor=west]  at (17,3) {\large $UC_2(6) = 0001$};
\end{tikzpicture}

With $i =6$: 

\begin{tikzpicture}[scale = 0.5]
\draw (0,0) grid[xstep=1,ystep=1] (15,7);
\fill (0,0) rectangle (2,1);
\fill (4,0) rectangle (5,1);
\fill (2,1) rectangle (4,2);
\fill (6,1) rectangle (7,2);
\fill (5,2) rectangle (7,3);
\fill (9,2) rectangle (10,3);
\fill (6,3) rectangle (8,4);
\fill (10,3) rectangle (11,4);

   \node [anchor=west] at (17,4) {\large $(A+B)_{0-2} = 11111111011$};
      \node[anchor=west]  at (17,3) {\large $UC_2(8) = 011$};
\end{tikzpicture}

$\vdots$

With $i=10$:

\begin{tikzpicture}[scale = 0.5]
\draw (0,0) grid[xstep=1,ystep=1] (15,7);
\fill (0,0) rectangle (2,1);
\fill (4,0) rectangle (5,1);
\fill (2,1) rectangle (4,2);
\fill (6,1) rectangle (7,2);
\fill (5,2) rectangle (7,3);
\fill (9,2) rectangle (10,3);
\fill (6,3) rectangle (8,4);
\fill (10,3) rectangle (11,4);
\fill (8,4) rectangle (10,5);
\fill (12,4) rectangle (13,5);
\fill (9,5) rectangle (11,6);
\fill (13,5) rectangle (14,6);
\fill (10,6) rectangle (12,7);
\fill (14,6) rectangle (15,7);

   \node [anchor=west] at (17,4) {\large $(A+B)_{0-2} = 111111111111111$};
      \node[anchor=west]  at (17,3) {\large $UC_2(15) =$ empty word};
\end{tikzpicture}

The next step gives $i = 15 = \max B + \max A +1$.

\end{example}

\begin{remark} Such a tiling is not unique: the rhythmic pattern $A = \left\lbrace 0,1,4\right\rbrace $ tiles with  $B=\left\lbrace 0,2,5,6,8,9,10\right\rbrace $, but also with the following $B_i$:
\begin{itemize}
\item $B_1 = \left\lbrace 0,2,5,6,8,9,10,   15,17,20,21,23,24,25\right\rbrace $
\item $B_2 = \left\lbrace 0,2,5,6,8,9,10,   15,17,20,21,23,24,25,   30,32,35,36,38,39,40\right\rbrace $
\item $\ldots$
\end{itemize}
\end{remark}

\begin{proposition}
This greedy algorithm is optimal in the sense that it gives the smallest number of entries $\sharp B$ (and so the smallest $N$) such that $(A,B)$ is a compact RT$_p$C of $\mathbb{Z}_N$.

\end{proposition}

\begin{proof}

Let $p$  be a prime number, and $A$ a rhythmic pattern such that it admits a compact RT$_p$C. \textit{Reductio ad absurdum}, let us suppose there exists $\widetilde{B}$ such that $(A,\widetilde{B})$ is a compact RT$_p$C of $\mathbb{Z}_{\widetilde{N}}$ and $\sharp \widetilde{B} < \sharp B$ with $B$ the output of the algorithm.

$\widetilde{B} \nsubseteq B$ because since $(A,\widetilde{B})$ is a compact RT$_p$C, it means that during the algorithm, the smallest element $i$ of $B \setminus \widetilde{B}$ would verifie line $4$: $i = \max \widetilde{B} + \max A +1$ and hence $B = \widetilde{B}$.

We can suppose without loss of generality that $B$ and $ \widetilde{B}$ are sorted. Let $i$ the smallest index such that $\widetilde{B}\left[ i\right]  = \widetilde{b} < B\left[ i\right]  = b$. 

After $i-1$ loops, the algorithm produces the subset $B\left[ 0,\ldots i-1\right] $ providing the same under-cover $UC_p (0)$ as the subset $\widetilde{B}\left[ 0,\ldots i-1\right] $: $$UC_p (0)= (A + B\left[ 0,\ldots, i-1\right] )_{0-p} \left[ 0,\ldots, N-1\right] =( A +\widetilde{B}\left[ 0,\ldots,i-1\right])_{0-p} \left[ 0,\ldots, N-1\right].$$ 

Since $\widetilde{B}$ tiles with $A$, at time $\widetilde{b}$, there cannot be exactly $1$ mod $p$ onset in the under-cover: $UC_p(0)\left[ \widetilde{b}\right] \neq 1 \text{ mod } p$. Step $3$ of the algorithm, during the $i$-th loop, returns the smallest index $b'$ such that at this index of the under-cover there is anything but $1$ mod $p$ onset, and will give a $b' \leq \widetilde{b} < b$, which is a contradiction with the fact that the $i$-th element of $B$ is $b$.

\end{proof}

\begin{example}
As we can see in example \ref{A014}, the smallest number of entries to tile modulo 2 with $A = \left\lbrace 0,1,4\right\rbrace $ is 7.
\end{example}

\begin{remark}
Thanks to theorem \ref{vuzaconcat}, we have seen that VC are the basis of RTC under the concatenation operation. Also, corollary \ref{t2vc} gives us the result that VC are the key to attaining a necessary and sufficient condition on the tiling property of a rhythmic pattern.

This greedy algorithm provides minimal RT$_p$C under concatenation. They are the ``modulo $p$ Vuza canons". We  hope that a connection can be established between both kinds of Vuza canons, the latter modulo $p$ being relatively easy to compute, to be able to return to classical VC and thus obtain a faster way to enumerate them. It remains, however, a distant ambition. There are two principal differences between classical VC and the RT$_p$C ones we are working with. First, the algorithm only computes compact RT$_p$C, whereas VC are not usually compact.  Future research, with the aim of linking back those two types of canons, will have to search and compute minimal non compact mod $p$ Vuza canons, i.e. minimal under concatenation and duality, but also tiling a smaller period than the compact one. The second difference is of course the distinction between tiling $\mathbb{Z}$ and tiling $\mathbb{F}_p$. A VC admits one and only one onset every beat, unlike RT$_p$Cs which admit one modulo $p$ per beat. So we need to understand how to remove those extra onsets. The paper will now focus on a better understanding of this latter challenge.

The goal is to understand when and how we are able to obtain `\textit{donsets}', a neologism introduced here to express the fact that the onsets have copies at this time, giving us a covering instead of a tiling. This term is also a contraction of $d$-onset, and if we have 3 layered notes, we can talk about a $3$-onset. 

\begin{definition}
We note the \textit{multiset of donsets} of the RT$_p$C $(A,B)$ of $\mathbb{Z}_N$: 
$$D_{(A,B),p} =  \bigcup_{n\in \mathbb{Z}_N}  \bigcup\limits_{\substack{kp, \\ k \in \mathbb{N}^*, \\ \mathds{1}_{(A+_N B)}(n) = kp+1}} \left\lbrace n  \right\rbrace$$
\end{definition}

\end{remark}

One can extend the notion of duality to compact RT$_p$C:

\begin{proposition}

$(A,B)$ is a (resp. compact) RT$_p$C iff $(B,A)$ is a (resp. compact) RT$_p$C and $D_{(A,B),p} = D_{(B,A),p}$. 

\end{proposition}

\begin{proof}
Let $(A,B)$ be a (resp. compact) RT$_p$C of $\mathbb{Z}_N$. $C = A+_N B $ (resp. $C = A+B$) is a multiset such that: $$C = \left\lbrace\underbrace{0,0,\ldots,0}_{k_0 p +1\text{ times}}, \underbrace{1,1,\ldots,1}_{k_1 p +1\text{ times}}, \ldots, \underbrace{N-1,N-1,\ldots,N-1}_{k_{N-1} p +1\text{ times}} \right\rbrace $$ with $\forall i \in \llbracket 0,N-1 \rrbracket, \ k_i \in \mathbb{N}$.

By commutativity in $\mathbb{Z}$, we get that $C = B+_NA$ (resp. $C = B+A$), hence $(B,A)$ is a (resp. compact) RTC of $\mathbb{Z}_N$ modulo $p$, and the donsets correspond to the $k_i \neq 0$ and are at the same positions $a+b \text{ mod } N = b+a\text{ mod } N$ (resp. $a+b = b+a$) and in the same amount $k_i p$.
\end{proof}

\begin{remark}
However, if $B$ is the smallest rhythmic pattern obtained by the algorithm for a given $A$, $A$ may not be the smallest complement obtained by the algorithm if the entry is $B$, if $A$ is a concatenation of different rhythmic patterns that tile with $B$. 

For instance, if the  rhythmic pattern in entry is $A = \left\lbrace 0,1,2,6,7,8\right\rbrace $, the algorithm modulo $2$ returns $B = \left\lbrace 0,3\right\rbrace $ for which it returns $A' = \left\lbrace 0,1,2\right\rbrace $. In such cases, we always have that such an $A'$ is always a subset of $A$.
\end{remark}

\begin{remark}
If the tiling is not compact, the best way we currently have to compute the size of a RTC quickly is the algorithm presented in \citep{amiotPNM49} which returns a multiple of the size of the smallest RTC with a given rhythmic pattern; this is achieved by incrementing the suitable values of $N$ until we can find one such that $A(X)$ divides $X^N -1$ modulo $p$.
One must bear in mind that in order to discover the exact size of the smallest RTC, one needs the greedy algorithm when the tiling is compact.

See, in Table \ref{01nN} the minimal size of RT$_2$C (hence compact) of rhythmic pattern $A = \left\lbrace 0,1,n\right\rbrace$ for some small $n$ obtained with the greedy algorithm.

One can notice that rhythmic patterns of the form $A_k = \left\lbrace 0,1,2^k\right\rbrace $ give smaller compact RTC. We will now focus on these patterns.

\begin{table}
\begin{center}

$\begin{array}{|c||c|c|c|c|c|c|c|c|c|c|c|c|c|c|c|}
\hline
n & \cellcolor{lightgray}2&3&\cellcolor{lightgray}4&5&6&7&\cellcolor{lightgray}8&9&10&11&12&13&14&15&\cellcolor{lightgray}16\\
\hline
N&\cellcolor{lightgray}3&7&\cellcolor{lightgray}15&21&63&127&\cellcolor{lightgray}63&73&889&1533&3255&7905&11811&32767&\cellcolor{lightgray}255\\
\hline
\end{array}$
%\captionsetup{justification=centering}
 \caption{\label{01nN} Size $N$ of the minimal compact RT$_2$C with the rhythmic pattern $A = \left\lbrace 0,1,n\right\rbrace $.}
 
\end{center}

\end{table}

\end{remark}

\section{The case $A_k = \left\lbrace 0,1,2^k\right\rbrace $}

The following theorem is the main new result of this paper, and is the first cardinality result for RT$_p$C.

\begin{theorem}\label{bg}
For all $k \in \mathbb{N}^*$, the rhythmic patterns couple $(A_k, B_k)=(\left\lbrace 0,1,2^k\right\rbrace, B_k)$ is the smallest compact RT$_2$C of $\mathbb{Z}_N$, and is of size $N = 4^k -1$ with $\sharp B_k = 4^k - 3^k$ entries, and $\sharp D_{(A_k,B_k),2} = 4^k - \dfrac{3^{k+1}-1}{2}$ donsets.
\end{theorem}

Thanks to a constructive proof, this theorem allows us to build the pattern of entries $B_k$ which tiles modulo 2 with a given $A_k = \left\lbrace 0,1,2^k\right\rbrace $. Since one can also easily prove that in this case one cannot have $d$-onsets with $d>3$, one arrives at a comprehensive enumeration: the number of entries, the size of the RTC, and the number of donsets.

Before giving a formal proof, we will start with an explanation of how those tilings are built, which will in turn demonstrate the general concept of the proof. Some technical lemmas will form the second part of this proof, and we will finish with the formal demonstration of the theorem.

\subsection{Introduction}

Let us begin with a close look at how the greedy algorithm builds a tiling with the pattern $A_3 = \left\lbrace 0,1,8\right\rbrace $.

At step $1$, the algorithm places the pattern $A_3$ at the first position, that is to say $B = \left\lbrace 0\right\rbrace $, hence $(A+B)_{0-2} = 110000001$.

The algorithm then enters the \textbf{while} loop, and tries first to fill the under-cover $UC_2(2) = 0000001$. It is filled with the entries $B' = \left\lbrace 0,2,4\right\rbrace $ which gives after three steps 3 of the algorithm $B= B \cup (\left\lbrace 0,2,4\right\rbrace +\left\lbrace 2\right\rbrace ) = \left\lbrace 0,2,4,6\right\rbrace $, thus $(A +B)_{0-2} = 111111111010101$. The algorithm goes again in the while loop, trying now to fill the under-cover $UC_2(9) = 010101$. And so on.

The idea is to fill the current under-cover, and doing so, we get a new under-cover: the ``following under-cover" of definition \ref{follow} ; and we keep on filling it until we obtain a compact tiling, that is to say an empty or null under-cover. The entries pattern $B_k$ is built in this incremental way, filling the following under-covers one after another during the \textbf{while} loop. 

\begin{remark}\label{remplissagetheo}
Let $i \in \mathbb{N}$, our rhythmic pattern $A_k = \left\lbrace 0,1,2^k\right\rbrace$ being of size $max(A_k) - min(A_k) = 2^k$, if, to fill any under-cover of size $2^k$ starting at index $i$ (and ending at index $i+2^k -1$) we need the entries $B'$, then the following under-cover of size $2^k$ (and you can find one of this size) and starting at index $i+2^k$ will be $B'_{0-1}$. This fact is due to the gap between the first onsets $\left\lbrace 0,1\right\rbrace $ and the third and last onset $\left\lbrace 2^k\right\rbrace $ of $A_k$.
\end{remark}

\begin{example}
Let $A_2 = \left\lbrace 0,1,4\right\rbrace $ and $B = \left\lbrace 0,2\right\rbrace $ which fills the null under-cover $0000$ starting at index $0$. We then obtain $(A_2 + B)_{0-2} = 11111010$, hence the following under-cover starting at index $2^2 +0 = 4$ and of size $4$ is $1010 = B_{0-1}$.

\begin{center}
\includegraphics[scale=0.75]{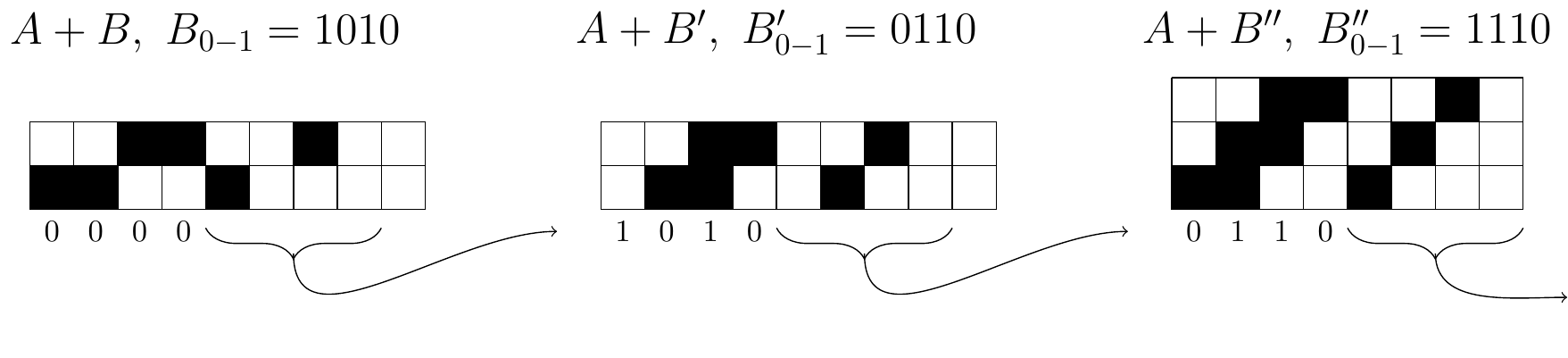}
\end{center}

\end{example}

An efficient way to exploit this phenomenon is to built $B_k$ in a recursive way. If we write $(B_k)_{0-1}$ in an array $V(k)$ with $2^k$ columns, from left to right then top to bottom, we obtain that every line $\ell$ of $V(k)$ is the following under-cover of the under-cover obtained with $A_k$ and $(B_k)_{0-1}$ read in lines $1\ldots \ell -1$.

\begin{example}\label{abvk}

For a better understanding, here are $A_k$, $B_k$, $(B_k)_{0-1}$ and $(B_k)_{0-1}$ in the $V(k)$ array for $k = 1,2,3$:

\begin{itemize}
\item[$\bullet$] $k=1, A_1 = \left\lbrace 0,1,2\right\rbrace , B_1 = \left\lbrace 0\right\rbrace, (B_1)_{0-1}= 1$

$V(1) = (10)$

\item[$\bullet$] $k=2, A_2 = \left\lbrace 0,1,4\right\rbrace , B_2 = \left\lbrace 0,2,5,6,8,9,10\right\rbrace,$

$(B_2)_{0-1} = 10100110111$

$V(2) = \left( \begin{array}{cccc}
1&0&1&0\\
0&1&1&0\\
1&1&1&0
\end{array}\right) $

\item[$\bullet$] $k=3, A_3 = \left\lbrace 0,1,8\right\rbrace$, \begin{eqnarray*}
B_3 &= &\left\lbrace 0, 2, 4, 6, 9, 10, 13, 14, 16, 17, 18, 20, 21, 22, 27, 28, 29, 30, 32,\right. \\
&& \left. 34, 35, 36, 37, 38, 41, 42, 43, 44, 45, 46, 48, 49, 50, 51, 52, 53, 54\right\rbrace
\end{eqnarray*}
$(B_3)_{0-1} = 10101010011001101110111000011110101111100111111011111110$

$V(3) = \left( \begin{array}{cccccccc}
1&0&1&0&1&0&1&0\\
0&1&1&0&0&1&1&0\\
1&1&1&0&1&1&1&0\\
0&0&0&1&1&1&1&0\\
1&0&1&1&1&1&1&0\\
0&1&1&1&1&1&1&0\\
1&1&1&1&1&1&1&0
\end{array}\right) $

\end{itemize}

\end{example}

\begin{definition}

Let $k\in \mathbb{N}^*$, $B$ a pattern, and $UC_2$ an under-cover of $(A_k,B)$ of size $n \leq 2^k$. We denote $\overset{k}{\longrightarrow}$ a postfix function over and to the finite words over the alphabet $\left\lbrace 0,1\right\rbrace $ defined by $UC_2 \overset{k}{\longrightarrow} (B')_{0-1}$ iff the pattern $B'$ of size $n$ (you can pad in some $0$s in the end to adjust the size) fills $UC_2$ with $A_k$.

%If to fill an under-cover $UC_2$ of $(A_k,B)$ of size $n \leq 2^k$ we need the pattern $B'$ of size $n$ (adding some ), we note $UC_2 \overset{k}{\longrightarrow} (B')_{0-1}$, where $\overset{k}{\longrightarrow}$ is a postfix function defined over and to the finite words over the alphabet $\left\lbrace 0,1\right\rbrace $.

\end{definition}

\begin{example}
We have $\overline{0}^{2n} \overset{k}{\longrightarrow} \overline{10}^n$ if $2n\leq 2^k$.

See in example \ref{abvk} above, for $k = 3$ and $n= 4$, one can see that the following under-cover of $\overline{0}^8$ of size $8$ (which is the first under-cover of the tiling) is $\overline{10}^4$.
\end{example}

\begin{remark}

Thanks to remark \ref{remplissagetheo}, we know that the pattern $B_k$ such that $(A_k, B_k)$ is a RT$_2$C can be written in the form $(B_k)_{0-1} = \tilde{B_1}\cdots  \tilde{B_n}$ with $$\forall i \leq n \  \vert\tilde{B_i} \vert= 2^k \text{ and } \overline{0}^{2^k} \overset{k}{\longrightarrow} \tilde{B_1} \overset{k}{\longrightarrow}  \ldots \overset{k}{\longrightarrow} \tilde{B_n}$$ and the theorem's proof will yield $n = 2^k-1$. It means that every under-cover $\tilde{B_i}$ will be the following under-cover of the under-cover $\tilde{B}_{i-1}$ as defined in \ref{follow}.

Written in the array $V(k)$ as defined earlier, it means that $V(k) =  \left( \begin{array}{c}
\tilde{B_1}\\
\tilde{B_2}\\
\cdots\\
\tilde{B_n}
\end{array}\right) $.

\end{remark}

\begin{example}
With $k=2$, we have  $\overline{0}^4 \overset{2}{\longrightarrow} 1010 \overset{2}{\longrightarrow} 0110 \overset{2}{\longrightarrow} 1110$ and indeed we have $(B_{2})_{0-1} = 10100110111 = 101001101110$ that is to say: 

$V(2) =  \left( \begin{array}{c}
1010\\
0110\\
1110
\end{array}\right) $.

\end{example}

\begin{definition}\label{tkconstruction}
Let us define the array $T(k)$ by induction over  $k\in \mathbb{N}^*$:

Let $T(1) = \left( 10\right) $. If one has the array $T(k)$ of size $(2^k-1,2^k)$ , one can build the array $T(k+1)$ of size $(2^{k+1}-1, 2^{k+1})$ in this way:

$\left( \begin{array}{c|c c}
T(k) & T(k) &\\
\hline
& 11\ldots 1& 0\\
& 11\ldots 1& 0\\
\widetilde{T(k)} & \vdots& \vdots\\
&11\ldots 1&0

\end{array}\right) $

with $\widetilde{T(k)}$ being an array of size $(2^k,2^k)$ defined by: its first line is $\overline{0}^{(2^k-1)}1$, and its last lines are $T(k)$ whose last column of $0$ is changed into a column of $1$ (one can easily prove by construction that the last column of $T(k)$ is null for all $k$).

\end{definition}

\begin{example}
See for some values of $k$ the array $T(k)$ with separation lines from the inductive construction:

$T(1) = (10)$

$T(2) = \left( \begin{array}{cc|cc}
1&0&1&0\\
\hline
0&1&1&0\\
1&1&1&0
\end{array}\right) $

$T(3) = \left( \begin{array}{cccc|cccc}
1&0&1&0&1&0&1&0\\
0&1&1&0&0&1&1&0\\
1&1&1&0&1&1&1&0\\
\hline
0&0&0&1&1&1&1&0\\
1&0&1&1&1&1&1&0\\
0&1&1&1&1&1&1&0\\
1&1&1&1&1&1&1&0
\end{array}\right) $

where we can see $\widetilde{T(2)} = \left( \begin{array}{cccc}
0&0&0&1\\
\hline
1&0&1&\textit{1}\\
0&1&1&\textit{1}\\
1&1&1&\textit{1}
\end{array}\right) $

\end{example}

\begin{remark}

The proof of the theorem will be to establish the equality $V(k) = T(k)$ for all $k \in \mathbb{N}^*$. For this purpose we need some technical lemmas presented in the next part of this proof.

\end{remark}

\subsection{Lemmas}

\begin{lemma}\label{recu}
For all $n \in \mathbb{N}^*$, to fill an under-cover $UC_2(i)_n = 0\overline{1}^n0$ starting at index $i$ with the pattern $ 11$, one needs the entries $B'_{0-1} = \overline{1}^{n+1}0$.
\end{lemma}

\begin{proof}
By induction over $n$.

If $n =1$, to fill the under-cover $UC_2(i)_1 = 010$  with the pattern $11$, one needs to put an onset at index $i$, hence $B'_{0-1} = 1\centerdot\centerdot$. The new under-cover is $100 = 1(1+1 \text{ mod } 2)0$. One then needs to put the pattern at index $i+1$ to fill the new $0$, hence $B'_{0-1} = 11\centerdot$. The new under-cover being $111$, the under-cover is a tiling and we do not need any more entry, hence $B'_{0-1} = 110$.

Assume the lemma is true for some $n$, let us see how one can fill the under-cover $UC_2(i)_{n+1} =0\overline{1}^{n+1}0$ starting at index $i$ with the pattern $11$.

The under-cover starts with a $0$, hence the word of the filling pattern starts with a $1$, i.e. $B'_{0-1} =1 \overline{\centerdot}^{n+2}$, and the new under-cover is $10\overline{1}^n0 = 1(1+1 \text{ mod } 2)\overline{1}^n0 = 1\cdot UC_2(i+1)_{n}$. By induction hypothesis, to tile the $n+2$ last index of this under-cover, the $n+2$ last terms of the entries $B'_{0-1}$ have to be $\overline{1}^{n+1}0$. Hence we have $B'_{0-1} = 1\overline{1}^{n+1}0 = \overline{1}^{n+2}0$.
\end{proof}

\begin{lemma}\label{doubletaille}
$\forall k \in \mathbb{N}, \overline{1}^{(2^k-1)}0\ \overline{1}^{(2^k-1)}0 \overset{k+1}{\longrightarrow} \overline{0}^{(2^k-1)}1\overline{1}^{(2^k-1)}0$

\end{lemma}

\begin{proof}
 Let us denote $\overline{1}^{(2^k-1)}0\ \overline{1}^{(2^k-1)}0 \overset{k+1}{\longrightarrow} \tilde{C}$.

The first block of $1$s of size $2^k -1$ is already a tile, there is no need of entries for those $2^k -1$ first indexes, hence the word $\tilde{C}$ starts with $\overline{0}^{2^k-1}$. 

We then arrive to the second block $0\ \overline{1}^{(2^k-1)}0$, whose length verifies $\vert 0\ \overline{1}^{(2^k-1)}0 \vert <2^{k+1}$. Only the first two onsets of $A_{k+1}= \left\lbrace 0,1,2^{k+1}\right\rbrace $ are used to tile the under-cover, hence we try to tile the rest of $\tilde{C}$ with the pattern $\left\lbrace 0,1\right\rbrace _{0-1} = 11$. So the filling verifies  the conditions of lemma \ref{recu}, meaning $\tilde{C}\left[ 2^{k} \ldots 2^{k+1}\right] = \overline{1}^{2^k}0$.

We indeed have $\tilde{C} = \overline{0}^{(2^k-1)}1\overline{1}^{(2^k-1)}0$.

\end{proof}

\begin{lemma}\label{2empartieremplie}
$\forall k \in \mathbb{N}^*, \forall n < 2^k-2, 0 \overline{1}^{n}0 \overset{k}{\longrightarrow} \overline{1}^{(n+1)}0$.

\end{lemma}

\begin{proof}
Since $\vert  0 \overline{1}^{n}0\vert < 2^k$, we only tile with the first two onsets $11$ of the pattern $A_k$, and we are again in the conditions of lemma \ref{recu}.
\end{proof}

\begin{lemma}\label{fin}
$\forall k \in \mathbb{N}$, $0 \overline{1}^{(2^k-2)}0 \overset{k}{\longrightarrow} \overline{1}^{(2^k)}$

\end{lemma}

\begin{proof}

In this case, we don't have  $\vert 0 \overline{1}^{(2^k-2)}0  \vert < 2^k$ and we cannot use lemma \ref{recu}. However, we have exactly $\vert 0 \overline{1}^{(2^k-2)}0  \vert = 2^k$. It means that we are out of lemma \ref{recu} conditions only for the last index. Hence if we have to add a pattern at the first index (and since the under-cover $0 \overline{1}^{(2^k-2)}0$ starts with a $0$ we will have to), we can consider the same under-cover with an additional onset at its last index $0 \overline{1}^{(2^k-2)}1 $ that we tile with the pattern $11$.

By an induction like the one used to prove lemma \ref{recu}, we obtain the entries $\overline{1}^{(2^k-1)}0$. But the $0$ on the last index is already filled by the last onset of $A_k$, so we only need to add a last time the pattern $A_k$ at the last index, and the following under-cover is $\overline{1}^{(2^k)}$.

\end{proof}

\begin{lemma}\label{1alademifin}
$\forall k \in \mathbb{N}$, if $\vert A \vert= \vert B \vert= 2^k -1$, if $A\cdot 0 \overset{k}{\longrightarrow} B\cdot 0$ then $A\cdot 1 \overset{k}{\longrightarrow} B\cdot 1$, with $U\cdot u$ being the concatenation of the word $U$ to the letter $u$.

\end{lemma}

\begin{proof}
By definition of an under-cover, adding some zeros at the end does not change it, so $A = A\cdot 0$.
If we denote $\tilde{B}$ (resp. $\tilde{B1}$) the rhythmic pattern such that $A = A\cdot 0$ (resp. $A \cdot 1$) is the $UC_2(i)$ of $(A_k,\tilde{B})$ (resp. of $(A_k,\tilde{B1})$), and associate the words $A,B$ with the sets $\mathsf{A},\mathsf{B}$ we have by definition $$\left( A_k+\left(\tilde{B} \cup \left(\mathsf{B} + \left\lbrace i\right\rbrace \right)\right)\right)_{0-2}\left[ 0,\ldots, i+2^k-1\right]  = \overline{1}^i \overline{1}^{\vert A\cdot 0 \vert} \text{ and } (A_k + \tilde{B})_{0-2}= \overline{1}^i \cdot A\cdot 0.$$

Denoting by $C$ a multiset of the form $$C = \left\lbrace\underbrace{0,0,\ldots,0}_{2k_0  +1\text{ times}}, \underbrace{1,1,\ldots,1}_{2k_1  +1\text{ times}}, \ldots, \underbrace{i-1,i-1,\ldots,i-1}_{2k_{i-1}  +1\text{ times}} \right\rbrace $$ 
with $\forall n \in \llbracket 0,i-1 \rrbracket, \ k_n \in \mathbb{N}$ such that $A_k + \tilde{B} = C \cup \left(\mathsf{A} +\left\lbrace i\right\rbrace \right)$, we have 
$$\left( \left[ C \cup\left(\mathsf{A} +\left\lbrace i\right\rbrace \right)\right] \cup \left[ A_k + \left(\mathsf{B} +\left\lbrace i\right\rbrace \right)\right] \right) _{0-2}\left[ 0,\ldots, i+2^k-1\right]  = \overline{1}^i \overline{1}^{\vert A\cdot 0 \vert}$$

or, since $\vert B \vert = 2^k-1$, $$\left(  C \cup\left(\mathsf{A} +\left\lbrace i\right\rbrace \right) \cup  \left( \left\lbrace 0,1\right\rbrace  + \left(\mathsf{B} +\left\lbrace i\right\rbrace \right)\right)  \right) _{0-2}\left[ 0,\ldots, i+2^k-1\right]  = \overline{1}^i \overline{1}^{\vert A\cdot 0 \vert}$$

by removing the $i$-th first indexes of $\overline{1}^i \overline{1}^{\vert A\cdot 0 \vert}$ that come from $C$ $$\left(  \left(\mathsf{A} +\left\lbrace i\right\rbrace \right) \cup  \left( \left\lbrace 0,1\right\rbrace  + \left(\mathsf{B} +\left\lbrace i\right\rbrace \right)\right)  \right) _{0-2} \left[ 0,\ldots, i+2^k-1\right] = \overline{0}^i \overline{1}^{\vert A\cdot 0 \vert}$$

so, by excising the $i$-th first beats: $$\left(\mathsf{A} \cup  \left(\left\lbrace 0,1\right\rbrace  +\mathsf{B}\right)\right) _{0-2} \left[ 0,\ldots,2^k-1\right] = \overline{1}^{\vert A\cdot 0 \vert} = \overline{1}^{2^k}$$

When one fills the under-cover $A\cdot 0$, one needs the entries $B\cdot 0$. Likewise, since the sets associated with the words $A\cdot 0$, $B\cdot 0$ are the same as those associated with the words $A,B$ by definition, we have, using the same notation $$\left(\mathsf{A\cdot 0} \cup \left(\left\lbrace 0,1\right\rbrace  +\mathsf{B\cdot 0}\right)\right) _{0-2} \left[ 0,\ldots, 2^k-1\right] = \overline{1}^{\vert A \vert+1}$$ 

The sets $\mathsf{A\cdot 1}$, $\mathsf{B\cdot 1}$  associated with the words $A\cdot 1$, $B\cdot 1$ verify $\mathsf{A\cdot 1} = \mathsf{A} \cup \left\lbrace 2^k\right\rbrace $, $\mathsf{B\cdot 1} = \mathsf{B} \cup \left\lbrace 2^k\right\rbrace $.

We have

\begin{eqnarray*}
\left(\mathsf{A\cdot 1} \cup  (\left\lbrace 0,1\right\rbrace  +\mathsf{B\cdot 1})\right) &=&\left((\mathsf{A} \cup \left\lbrace 2^k\right\rbrace)\cup  (\left\lbrace 0,1\right\rbrace  +(\mathsf{B}\cup \left\lbrace 2^k\right\rbrace))\right) \\
&=&\left(\left\lbrace 2^k\right\rbrace\cup \mathsf{A}  \cup  (\left\lbrace 0,1\right\rbrace  +\mathsf{B} )\cup( \left\lbrace 2^k,2^k+1\right\rbrace))\right) \\
&=&\left(\left(  \mathsf{A}  \cup  (\left\lbrace 0,1\right\rbrace  +\mathsf{B} )\right) \cup( \left\lbrace 2^k,2^k,2^k+1\right\rbrace))\right)
\end{eqnarray*}

Thus $$\left(\mathsf{A\cdot 1} \cup  (\left\lbrace 0,1\right\rbrace  +\mathsf{B\cdot 1})\right) _{0-2} =\overline{1}^{2^k-1}(1+1+1)1 = \overline{1}^{\vert A \vert+1}$$

So, by adding the first $i$ beats with a multiset $C1$ such that $A_k + \tilde{B1} = C1 \cup \left(\mathsf{A\cdot 1} +\left\lbrace i\right\rbrace \right)$, we have

%using the same $C$ to add again the first $i$ beats to those $\mathsf{A\cdot 1}$, $\mathsf{B\cdot 1}$ sets, we have

$$\left( \left[ C1 \cup\left(\mathsf{A\cdot 1} +\left\lbrace i\right\rbrace \right)\right] \cup \left[ A_k + \left(\mathsf{B\cdot 1} +\left\lbrace i\right\rbrace \right)\right] \right) _{0-2}\left[ 0,\ldots, i+2^k-1\right]  = \overline{1}^i \overline{1}^{\vert A\cdot 1 \vert}$$

which means

$$\left( A_k+\left(\tilde{B1} \cup \left(\mathsf{B\cdot 1} + \left\lbrace i\right\rbrace \right)\right)\right)_{0-2}\left[ 0,\ldots, i+2^k-1\right]  = \overline{1}^i \overline{1}^{\vert A\cdot 1 \vert},$$

 we indeed have that  $A\cdot 1 \overset{k}{\longrightarrow} B\cdot 1$.

\end{proof}

\begin{lemma}\label{repetition}
$\forall n,k \in \mathbb{N}, \forall a \in \lbrace 0,1\rbrace$, if $\vert A\vert = \vert B\vert = 2^k -1$, if $A\cdot a \overset{k}{\longrightarrow} B\cdot 0$ then $\overline{A\cdot a}^n \overset{k+n -1}{\longrightarrow} \overline{B\cdot 0}^n$. 

\end{lemma}

\begin{proof}
The fact that $B\cdot 0$ ends with a $0$ means that to fill $A\cdot a = UC_2(i)$, one does not  need to put the pattern $11$ (first two onsets of $A_k$) at index $i+2^k$. Hence we do not brim over the $2^k$-sized following under-cover when filling  $A\cdot a$ with $B$. To fill the following under-cover we do not need to keep the memory of the previous filling (except for the last onset of $A_k$ which creates the following under-cover of course). We can concatenate the filling pattern if we want to fill the $n$-concatenation of the under-cover, taking care to tile with $\left\lbrace 0,1,2^{k+n-1}\right\rbrace $  instead of $\left\lbrace 0,1,2^{k}\right\rbrace $ to make sure that we will not run into the last onset of $A_{k+n-1}$ and  that we will only use the first two onsets $11$ during the filling.
\end{proof}

Those previous lemmas help for the recursive proof of the equality $V(k) = T(k)$, whereas the next one will be used for the cardinality result.

\begin{lemma}\label{konsets}
If one constructs a minimal RT$_2$C with $A$ such that $\sharp A = n$, then every $k$-onset verifies $k\leq n$.

\end{lemma}

\begin{proof}
Let $A = \left\lbrace a_1, \ldots, a_n \right\rbrace $, $B$ such that $(A , B)$ is a RT$_2$C of $\mathbb{Z}_N$ and suppose a $k$-onset at time $t$ with $k>n$.

This $k$-onset is built from different layered onsets of $A$: $a_{j_1}, \ldots a_{j_k}$, and we note $b_{j_i}$ the voice entry in $B$ that produces the onset $a_{j_i}$ played at time $t$.

Since $k>n$, by the pigeonhole principle, there exists $j_i \neq j_l$ such that $a_{j_i} = a_{j_l}$. Hence the rhythmic pattern is played twice at the same entry $b_{j_i} = b_{j_l} = t - a_{j_i}$. So, at every time $b_{j_i} + a_m = b_{j_l} + a_m$, $m=1,\ldots,k$, there is two onsets layered from $A$, that counts as $0$ mod $2$.

So, if we will have the exact same number of onsets mod $2$ at every time of $A +_N B$ and $A +_N (B \setminus \left\lbrace b_{j_i}, b_{j_l} \right\rbrace ) $.

Hence $(A, B \setminus \left\lbrace b_{j_i}, b_{j_l} \right\rbrace)$ is a smaller RT$_2$C of $\mathbb{Z}_N$, which is a contradiction with the fact that $(A , B)$ is minimal.
\end{proof}

\subsection{Proof of the Theorem}

To prove theorem \ref{bg}, we will start to prove the

\begin{theorem}
The array $T(k)$ built previously in definition \ref{tkconstruction} is equal to the array $V(k) = \left(  \begin{array}{c}
\tilde{B_1}\\
\tilde{B_2}\\
\cdots\\
\tilde{B_n}
\end{array}\right) $

with $B_k = \tilde{B_1}\cdots \tilde{B_n}$, i.e. $\overline{0}^{2^k} \overset{k}{\longrightarrow} \tilde{B_1} \overset{k}{\longrightarrow}  \ldots \overset{k}{\longrightarrow} \tilde{B_n}$ and $\vert \tilde{B_i} \vert = 2^k$.
\end{theorem}

\begin{proof}
Let us prove this equality by induction over $k$:

If $k=1$, we have indeed $T(1) = (10) = V(1)$ and the pattern of entries is $B_1 = \left\lbrace 0\right\rbrace $.

\vspace{0.5cm}
If the equality is true for some $k$, let us prove it for $k+1$:

\vspace{0.5cm}
By lemma \ref{repetition} with $n=2$ we obtain that if $T(k)$ gives the pattern of entries to tile with  the pattern $A_k$, then $(T(k) | T(k))$ gives the begining of the pattern of entries to tile with $A_{k+1}$. Indeed, we proceed by blocks of size $2^{k+1}$ before we come to the last onset of $A_{k+1}$.

\vspace{0.5cm}
By lemma \ref{doubletaille}, we obtain the construction of the first line of the bottom block of $T(k+1)$, i.e. the first line of $\widetilde{T(k)}$ which is $\overline{0}^{(2^k-1)}1$ adjoined to $\overline{1}^{2^k -1}0$.

\vspace{0.5cm}

By lemma \ref{2empartieremplie}, we obtain the block at the bottom-right of $T(k+1)$:

$\begin{array}{cc}
 11\ldots 1& 0\\
 11\ldots 1& 0\\
 \vdots & \vdots\\
11\ldots 1&0
\end{array}$

Indeed, one can notice that we always have a $0$ before those lines of $1$s, so we are in the condition of lemma \ref{2empartieremplie}; if we didn't, it means that we have a line full of $1$s, and hence we are in the case of lemma \ref{doubletaille} which entails having a complete tiling with $A_{k+1}$ and finishing the equality, see later when we use lemma \ref{fin} (or going towards step $k+2$).

\vspace{0.5cm}

Lemma \ref{1alademifin} explains why all but the first lines of $\widetilde{T(k)}$ look like $T(k)$. We try to tile with the pattern $11$, since the last onset of $A_{k+1}$ is not taken into account in a block of such size. With the recursion hypothesis $V(k) = T(k)$, the procedure gave us the word in $T(k)$, and we can think of it as though we were filling a previous line of $0$s, as it is highlighted in remark \ref{covernul}. Here the previous line is $\overline{0}^{2^k -1}1$, the first line of $\widetilde{T(k)}$, hence the following word obtained is $T(k)$ but with a $1$ instead of a $0$ in its last column as shown in lemma \ref{1alademifin}.

\vspace{0.5cm}
Finally, lemma \ref{fin} proves the end of construction. The last line of $T(k+1)$ is $\overline{1}^{(2^{k+1}-1)}$ which gives a tiling modulo 2 with $A_{k+1}=\left\lbrace  0,1,2^{k+1}\right\rbrace  $ because we use in last position (index $2^{k+1}$ of this line) the last onset of $A_{k+1}$. Indeed this sequence of  $2^{k+1} -1$ consecutive entries gives a tiling modulo 2 which is followed by  $2^{k+1} -1$ onsets not donsets, finishing the tiling.

\vspace{0.75cm}
We can now conclude by induction that $\forall k \in \mathbb{N}^*$ the array $T(k)$ built like in definition \ref{tkconstruction} is equal to $V(k)$, i.e. is the array containing $(B_k)_{0-1}$ written line by line.

\end{proof}

Now that we have proved this first result, we know how to construct the pattern of entries $B_k$, and in order to obtain the cardinality results in theorem \ref{bg}, we just have to tally in $T(k)$.

\begin{proof}

The number of entries $b(k) = \sharp B_k$ is the number of $1$s in the array $T(k)$ and we can count them by induction:

We have $b(1) = 1; b(2) = 7$.

By construction of $T(k+1)$, we have the equality $$b(k+1) = 3\times b(k) + 4^k.$$

Indeed, the term $3 \times b(k)$ comes from the array $T(k)$ put three times in $T(k+1)$ and the term $4^k$ is from the bottom right block of $1$s of size $(2^k,2^k)$.

One can verify easily that for all $k$, $b(k+2) - 7b(k+1) + 12b(k) = 0$, hence $$b(k) = \dfrac{b(2)-3b(1)}{4}\times 4^k - \dfrac{4b(1)-b(2)}{3} \times 3^4 = 4^k - 3^k.$$

\vspace{0.5cm}

To obtain the size $N$ of the RTC, we know that it will end with putting the pattern $A_k$ on the last $1$ of $(B_k)_{0-1}$, so $$N = (2^k-1)\times 2^k -1 + 2^k= 4^k -1.$$

The first product comes from the size $(2^k, 2^k -1)$ of the array $T(k)$, and the next $-1$ from the fact that the last $1$ in $T(k)$ is on the penultimate position. The last $+2^k$ is from the addition of the pattern $A_k$ at this index.

Finally, to enumerate the donsets, we simply use lemma \ref{konsets}. In our case, we have $\sharp A_k = 3$, so donsets cannot have more than $3$ onsets layered. And since we obtain a RT$_2$C, it means that there is $1$ or $3$ onsets at each time, so every donset is a $3$-onset.

Hence, if we note $D_k$ the set of index where there is a donset, we have the formula $$\sharp A_k \times \sharp B_k - 2\sharp D_k = N.$$

Indeed the factor $2$ multiplying $\sharp D_k$ comes from this lemma \ref{konsets}, saying that every donset is exactly two onsets above a classical tiling; the result follows: $\sharp D_k = 4^k - \dfrac{3^{k+1}-1}{2}$.

\end{proof}

\section{Conclusion}

Finding a way to obtain all the VC of a given period $\mathbb{Z}_N$ is an issue for composers. It is also a major challenge for mathematicians, as solving the Coven-Meyerowitz conjecture could help to solve the spectral conjecture in dimension 1. This conjecture, also called the Fuglede conjecture, has recently been proved to be false for dimensions higher or equal to 5 \citep{tao2003fuglede}, and from this work, \cite{matolcsi2005fuglede} has proved it to be false in dimensions 3 and 4. The problem remains open in smaller dimensions.

During the last decade, many approaches have been tested to built and understand VC: exhaustively \citep{KolMatoJMM}, by using perturbations of periodic tilings \citep{kolountzakis2003translational}, through transformations that conserve the property of being a VC \citep{amiot2005rhythmic}...

This article introduces a new approach with the modulo $p$ VC, and especially modulo $2$, which are substantially easier to compute. It is also a sub-problem of  a cryptographic issue: factorisation of polynomials in $\mathbb{F}_2\left[ X\right] $, which is equivalent to the factorisation of some multisets. The interested reader can find a recent work on factorisation of a multiset into two multisets \citep{neumann2004tensor}. The modulo $p$ VC problem can be rephrased as the still unresolved mathematical topic of factorisation of a multiset into two proper sets: a new mathematical branch to explore, rising from a musical bough.

\bibliographystyle{tMAM}
\bibliography{biblio.bib}

\end{document}